\documentclass{article}
\usepackage{a4wide}
\usepackage{amsfonts}
\usepackage[margin=2cm]{geometry}
\usepackage{graphics, amsmath,enumitem, comment, color}
\usepackage{graphicx}
\usepackage{epsfig}
\usepackage{amssymb}
\usepackage{amsmath,amsthm}
\usepackage{mathrsfs}
\usepackage{colortbl}
\usepackage{tikz}
\usetikzlibrary{positioning,arrows}
\usetikzlibrary{3d}
\usetikzlibrary[patterns]
\usepackage{pgflibraryarrows}
\usepackage{pgflibrarysnakes}
\usetikzlibrary{snakes}
\usepackage{arydshln}
\usepackage{graphicx}
\usepackage{setspace}
\usepackage{eucal}

\newif\ifcolorcomments

\colorcommentstrue
\def\bc{\begin{center}}
\def\ec{\end{center}}
\def\be{\begin{equation}}
\def\ee{\end{equation}}

\def\N{\mathbb N}
\def\Z{\mathbb Z}
\def\Q{\mathbb Q}

\def\R{\mathbb R}

\def\Z{\mathbb Z}

\newtheorem{lem}{Lemma}[section]

\newtheorem{pro}[lem]{Proposition}
\newtheorem{thm}[lem]{Theorem}

\newtheorem{con}[lem]{Conjecture}

\newtheorem{cor}[lem]{Corollary}
\newtheorem{rem}[lem]{Remark}
\numberwithin{equation}{section}
\allowdisplaybreaks
\usepackage{hyperref}

\title{On the minimum of $\sigma$-Brjuno functions}

\author{Ayreena Bakhtawar
\and
Carlo Carminati
\and
Stefano Marmi
}

\newcommand{\Addresses}{{
  \bigskip
  \footnotesize

A.~Bakhtawar, \textsc{
Institute of Mathematics, Polish Academy of Sciences, ul.  Sniadeckich 8, 00-656
Warszawa, Poland
}\par\nopagebreak
  \textit{E-mail address}, A.~Bakhtawar: \texttt{ abakhtawar@impan.pl; ayreena.bakhtawar@gmail.com}

  \medskip

  C.~Carminati, \textsc{Dipartimento di Matematica, University di Pisa, Largo Bruno Pontecorvo 5, 56127
Pisa, Italy}\par\nopagebreak
  \textit{E-mail address}, C.~Carminati: \texttt{carlo.carminati@unipi.it}

  \medskip

  S.~Marmi, \textsc{Scuola Normale Superiore, Piazza dei Cavalieri 7, 56126 Pisa, Italy}\par\nopagebreak
  \textit{E-mail address}, S.~Marmi: \texttt{stefano.marmi@sns.it}

}
}
\date{}
\begin{document}
\maketitle
\footnotetext{This research is supported by Centro di Ricerca Matematica Ennio de
Giorgi, Scuola Normale Superiore ``Research in Pairs Program'',  and by 
the Project ``
Dynamics and
Information Research Institute – Quantum Information (Teoria dell'Informazione), Quantum
Technologies" 
 within the agreement between UniCredit Bank and Scuola
Normale Superiore. The first author would like to thank Centro di Ricerca Matematica Ennio de
Giorgi, Scuola Normale Superiore, Pisa for their hospitality during various visits.}
\begin{abstract}

$\sigma$-Brjuno functions were introduced in \cite{MaMoYo_06} as an interesting variant of the classical Brjuno function, where one substitutes the $\log$ singularity at $x=0$ with the power law divergence $x^{-1/\sigma},$ $(\sigma>0).$ As in the classical case, $B_{\sigma}$ is a locally unbounded, highly irregular lower semi continuous function; from semi continuity property it easily follows that $B_{\sigma}$
admits a global minimum but to locate it is quite a challenging problem.
We prove that for $\sigma=n \in \mathbb{N}$, the unique global minimum of $B_n$ is achieved at the fixed point $ [0; \overline{n+1}]$.  Furthermore, we prove that these minimizers are locally stable, showing that the point of minimum remains constant for $\sigma$ in a neighborhood of $n$. Finally, we discuss the scaling behavior near these minima and we formulate a  conjecture about the phase transitions for the location of the minimizer as $\sigma$ varies. 

 \end{abstract}

\section{Introduction}

 The problem of linearizing a dynamical system near a fixed point stands as one of the foundational questions in complex dynamics. For holomorphic maps in one complex variable, this reduces to understanding when a germ $f(z)=\lambda z+O(z^2),$  where $\lambda:=f'(0),$ is locally conjugate to its linear part $z\mapsto\lambda z.$ When $|\lambda|\neq 1$ (the hyperbolic case), Koenigs’ theorem in the late 19th century provided a complete answer. The delicate indifferent case $|\lambda|=1,$ where $\lambda=e^{2\pi \iota x}$ and $x$ is irrational, proved far more subtle due to the appearance of small divisors.
 
Siegel’s 1942 \cite{Si_42},  breakthrough result established linearizability under a Diophantine condition on $x,$ a number-theoretic restriction on how well $x$ can be approximated by rationals. Two decades later, Brjuno \cite{Br_71} significantly relaxed this condition by introducing, what is now known as the Brjuno condition: for an irrational $x$ with continued fraction convergents  $\left\{    \frac{p_{j}}{ q_{j} }    \right\}_{j\geq 0},$ the Brjuno condition requires $\sum^{\infty}_{j=0}\frac{\log q_{j+1}}{q_{j}}<\infty.$
This defines the set of Brjuno numbers, a full-Lebesgue measure set, containing  all Diophantine numbers.
Similar results hold for the local conjugacy problem of analytic diffeomorphisms of the circle
\cite{Yo_02}, for problem of linearization of germs of analytic
diffeomorphisms with a fixed point, for 
 some complex area–preserving maps \cite{ChMa_22, Da_94, Mar_90}.
 
 Yoccoz's (see \cite[Appendix 5] {MaMoYo_01}) reformulated Brjuno  condition via the Brjuno function $B:\R\setminus\Q,$ constructed using the Gauss map 
$A:[0,1]\to[0,1]$ (defined in \eqref{Amap})    and satisfying a functional (cocycle) equation under the action of 
the modular group $PGL (2, \Z)$  as follows:  
\begin{equation*}
B(x)= \sum^{\infty}_{j=0} (xA(x)\cdots A^{j-1}(x))\log \left(\frac{1}{A^{j}(x)}\right).
\end{equation*}

Another important consequence of Yoccoz's work is that the Brjuno
function gives the size (modulo $L^{\infty}$ \cite{Yo_95}, and even continuous functions \cite{BC06}) of the domain of stability
around an indifferent fixed point of a quadratic polynomial. Conjecturally, it plays the same
role in many other small divisor problems \cite{MA_00,MaJa_92,  MaYa_02, MaCa08}.

The local properties of $B$ were studied in \cite{BaMa_12}, where the authors showed that the Lebesgue points of the Brjuno function $B$ are exactly the Brjuno numbers and the multifractal analysis of $B$ was carried out in \cite{StMa_18}. Recently, the local minima of $B$ was studied in \cite{BCM}.

A natural  generalization of this framework, proposed by Marmi--Moussa--Yoccoz  \cite{MaMoYo_06}, arises from replacing the logarithmic term with a function exhibiting  singular behavior at $x=0$; particularly interesting choice, as explored in  \cite{LuMar_10,MaMoYo_06}, is to substitute the logarithm with a power-law divergence, leading to the one-parameter family of $\sigma$-Brjuno functions, where $\sigma>0$. 
For $x \in (0,1)\setminus\mathbb{Q}$ with the Gauss map $A(x)=1/x-\lfloor 1/x\rfloor$ generating continued fraction iterates $x_j = A^j(x)$ and products $\beta_j = \prod_{i=0}^j x_i$, the $\sigma$-Brjuno function is defined as (see Figure \ref{mfig} for the graph of the Brjuno function for different values of $\sigma$),
\begin{equation*}\label{aBFun}
B_{\sigma}(x)=\sum^{\infty}_{j=0} \beta_{j-1}(x) x^{-1/\sigma}_{j}=\sum^{\infty}_{j=0} \beta_{j-1}(x)(1/A^{j}(x))^{1/\sigma}.
\end{equation*}
 
 This generalization establishes a precise hierarchy within Diophantine classes. As  detailed in Section 2, condition $B_{\sigma}<\infty$ is intimately related to classical Diophantine classes: the set of $\sigma$-Brjuno numbers contains all Diophantine numbers of exponent $\tau<\sigma$ and is contained in the set of Diophantine numbers of exponent $\sigma.$ Thus, the parameter $\sigma$ provides a continuous refinement of the classical Diophantine exponent $\tau.$

 Beyond this arithmetic chracterization, the function $B_{\sigma}$ exhibits rich variational properties. It is lower semi-continuous as shown in \cite{BCM} (see section 3.4 of \cite{BCM}) and blows up at every rational, $(\lim_{x\to p/q}B_\sigma(x)=+\infty).$  
 
 The lower semi-continuity implies that $B_\sigma$ always admits a minimum on $[0,1]$; as in the clssical case (see \cite{BCM}), it is then natural to ask  where the point of minimum of the function $B_{\sigma}$ is located for different values of $\sigma.$
  
Throughout this article $\N$ will denote the set of positive natural integers;
our investigation provides an exact answer for $\sigma\in\N$, while simultaneously establishing the local rigidity of these minima for nearby real values of the parameter.

\begin{thm}\label{MTV}
Let $\sigma=n \in \N.$ Then the $\sigma$-Brjuno function achieves its unique global minimum at the point 
\begin{equation*}
    \eta_{n+1}:=[0;\overline{n+1}]=\frac{\sqrt{(n+1)^2+4}-(n+1)}{2},
\end{equation*}
that is
\begin{equation*}
    \min_{x\in [0,1]} B_{n}(x)=B_{n}(\eta_{n+1}).
\end{equation*}
\end{thm}

The proof of this result will be carried in two steps which are enclosed in the propositions below:  an {\it a priori} estimate which shows that the global minimum must fall in the interval $[0,\frac{1}{n+1}]$ (Proposition \ref{Reg}); and an argument (analogous to that used by \cite{BaMa_20}) showing that if the minimum of $B_{\sigma}$ belongs to $[0,\frac{1}{n+1}]$ then it must coincide with $\eta_{n+1}$ (Proposition \ref{MT}). 

\begin{pro}\label{Reg}
Let $\sigma=n\in\N.$ Then  
\begin{equation}\label{war1}
\min_{x\in[\frac{1}{n},1]}B_{\sigma}(x)>B_{\sigma}(\eta_{n+1}).
\end{equation}
\end{pro}    
\begin{pro}
    \label{MT}
Let $n\in\N.$ If
 \begin{equation}\label{C22}
\sigma<n+\eta_{n+1}
\end{equation} 
 and
\begin{equation}\label{C1}
\min_{x\in[0,1]}B_{\sigma}(x)=B_{\sigma}(r) \text{ with }  r\leq \frac{1}{n+1}.
\end{equation}
Then 
$
r=\eta_{n+1}.$
\end{pro}   
\begin{rem}
    Proposition \ref{MT} actually holds also for $n=0$: in this case condition \eqref{C22} is automatically verified, and one gets that the point of minimum $r=\eta_1=\frac{\sqrt{5}-1}{2}$ for all $\sigma\in (0,\eta_1)$.
\end{rem}

\begin{proof}[Proof of Theorem \ref{MTV}]
    By the a priori estimate \eqref{war1}, any minimizer $r$ of $B_n$ satisfies $r\le 1/(n+1)$; moreover $\sigma=n$ trivially satisfies condition \eqref{C22} because $\eta_{n+1}>0$. Hence Theorem \ref{MT} applies and gives $r=\eta_{n+1}$.
\end{proof}

\begin{figure}[h]
 \centering
 \includegraphics[width=0.7\textwidth]{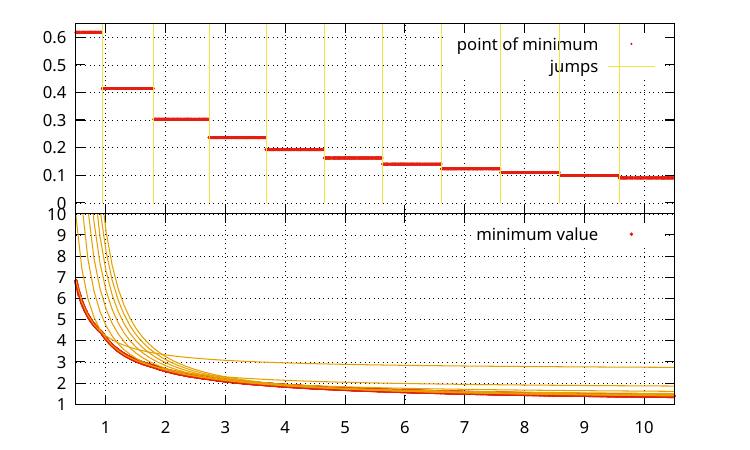} \label{minima1}
 \caption{   (Approximate) behaviour of the minimum point and minimum value as $\sigma$ ranges. The lower part shows the minimum value compared to the plot of the function
$\sigma \mapsto B_\sigma(\eta_n) $ for $1\leq n \leq 8$,
where $\eta_n=(\sqrt{n^2+4}-n)/2=[0; \overline{n}]$
is the n-th fixed point of the Gauss map. The point of minimum seems to undergo a jump at values $\sigma^*_n$ defined by equation \eqref{sstar}. 
 The yellow vertical lines in the upper part are drawn at the value $\sigma=\sigma^{*}_{n}.$}
\end{figure} 

Actually, numerical evidence  suggests an interesting phase transition phenomenon occurs: it seems that the minimizer does not drift continuously as $\sigma$ varies, but remains ``locked" at one of the fixed points of the Gauss map $A$, until it `jumps' to the next fixed point of $A$ when $\sigma$ reaches a critical threshold. This behaviour is quite evident from our numerical evidence (see Figure \ref{minima1}), and can be summarized as follows:

\begin{con}
    For all $n\in \N,$ there exists a value $\sigma^*_n \in (n-1,n)$ such that
    \begin{equation}\label{transition}
    \min_{x\in [0,1]} B_{\sigma}(x)=
   \begin{cases}
   B_{\sigma}(\eta_{n}) & \mbox{ if } \sigma \in [n-1,\sigma^*_n)\\ 
       B_{\sigma}(\eta_{n+1})& \mbox{ if } \sigma \in [\sigma^*_n , n].
   \end{cases}
    \end{equation}
    \end{con}
It is expected that  for all $n\in\N,$ this transition occurs at the point $\sigma_{n}^*$  which satisfies $B_{\sigma_{n}^*}(\eta_{n})=B_{\sigma^*}(\eta_{n+1})$, namely  
\begin{equation}\label{sstar}
\sigma^*_{n}=\frac{ \log\frac{\eta_{n}}  {\eta_{n+1}} } { \log\frac{1-\eta_{n+1}} {1-\eta_{n}}    }.
\end{equation}
One can compute the asymptotic expansion: $\sigma^*_{n}= n-\frac{1}{2}+\frac{5}{6n}+O(\frac{1}{n^2}),$ as $n\to\infty$; this tells that when $n>>1$ the jump approximately occurs at half integers. 
    
Even though our methods do not fully establish Conjecture \ref{transition}; Propositions \ref{Reg} and \ref{MT} allow us to prove
 that the point of minimum does not change when $\sigma$ ranges in a suitable neighbourhood (nbd) of $n.$

\begin{thm}\label{MTV2}
    For every $n \in \N,$ there exists an open interval  $U_n$ containing $n$ such that
    $$ \min_{x\in [0,1]} B_{\sigma}(x)=B_{\sigma}(\eta_{n+1}), \ \ \ \forall \sigma \in U_n.$$
\end{thm}
The proof of this result will follow immediately from Proposition \ref{MT} as soon as we show that there exists a suitable neighbourhood $U_n$ of $n$ such that the statement of Proposition \ref{Reg} holds for $B_\sigma$ for all $\sigma \in U_n$.

The paper is structured as follows. Section \ref{Aux}  collects background on continued fractions, auxiliary results and the connection of $\sigma$-Brjuno numbers with Diophantine numbers.  
Section \ref{APri} develops a priori estimates for general values of $\sigma$. We establish a universal lower bound and utilize cylinder-specific estimates to restrict the possible locations of the global minimum. 
Section \ref{GMin} is dedicated to the proofs of our main results. Here, we establish the main theorem, proving that the global minimum is uniquely achieved at the fixed point $\eta_{n+1}$ for any $\sigma$ in an interval $U_n$ containing the integer $n$. This section covers the localization and the stability of the minimizer, showing that the optimal arithmetic property of the fixed point is robust for $\sigma < n + \eta_{n+1}$. 
In Section \ref{SPn}, we provide the scaling analysis for these minima, using the case $n=2$ as a representative model for the general behavior.

\section{Definition and properties of $\sigma$-Brjuno functions.}\label{Aux}
This section establishes the necessary background on continued fractions and defines the $\sigma$-Brjuno function. We further explore its fundamental properties and its relationship with Diophantine approximation.
\subsection{Continued fraction preliminaries.}
 For $x\in(0,1)\setminus\Q,$
the regular continued fraction of a real number $x$ is associated to the Gauss map $A:(0,1)\to(0,1)$ defined by  \begin{equation}\label{Amap}
 A(x)=\frac{1}{x}-\left[  \frac{1}{x}  \right], \; \; \mbox{for } x\neq 0,
 \end{equation}
  where  $[x]$ denotes the integer part of $x.$ A more detailed description states that the map $A$ is made of the following branches:
 \begin{align*}
 A(x)&=\frac{1}{x}-k  \quad   \text{ for} \quad   \frac{1}{k+1} <x  \leq \frac{1}{k}. \end{align*}
  Setting
 $ x_{0}=x, a_{0}=0,$  and for any $n\geq 0,$ we have $x_{n+1}=A(x_{n})=A^{n+1}(x), \quad a_{n+1}=\left[\frac{1}{x_{n}} \right]\geq 1.$
 Then
  $x_{n}^{-1}=a_{n+1}+x_{n+1}  $  and we obtain the expansion 
 \begin{equation*}
 x=a_{0}+x_{0}=\frac{ 1}{a_{1}+ x_{1}}=\cdots=\frac{1 }{a_{1}+\frac{1}{a_{2}   +\cdots  +\frac{1 }{a_{n}+x_{n}}   }         }=[0;a_{1},\cdots, a_{n},\cdots].\end{equation*}
For all $n\in\N,$ the convergents
 $
 \frac{p_{n}}{q_{n}}=[0;a_{0};a_{1},\cdots,a_{n}            ],
 $ satisfy the  recursion relation
 \begin{equation*}
 p_{n}=a_{n}p_{n-1}+p_{n-2},\quad q_{n}=a_{n}q_{n-1}+q_{n-2},\quad p_{-1}={q_{-2}}=1,\quad p_{-2}=q_{-1}=0.\end{equation*}
with 
$q_{n}p_{n-1}-p_{n}q_{n-1}=(-1)^{n}$ and we can also write
\begin{align*}
x=\frac{p_{n}+p_{n-1}   x_{n}}{q_{n}+q_{n-1}x_{n}} \text{ and }
x_{n}=- \frac{q_{n}x-p_{n}}{q_{n-1}x- p_{n-1}}.
\end{align*}
 Let $\beta_{n}:=\prod^{n}_{i=0}x_{i}$ be the product of the iterates along the orbit $A$ satisfying 
 \begin{align*}
\beta_{n}=\prod^{n}_{i=0}x_{i}=(-1)^{n}(q_{n}x-p_{n}) \text{ for } n\geq0, \text{ with } \beta_{-1}=1,
 \end{align*}
 and  $x_{n}=\frac{\beta_{n}}{\beta_{n-1}}.$ The key estimate 
$ \frac{1}{2}<\beta_{n}q_{n+1}<1$ will be used repeatedly. 

\subsection{The $\sigma$-Brjuno function.}

For $\sigma > 0$, the $\sigma$-Brjuno function is defined by the series \begin{equation}\label{ps} B_{\sigma}(x) = \sum_{j=0}^{\infty} \beta_{j-1}(x) x_j^{-1/\sigma}. \end{equation} We denote the set of $\sigma$-Brjuno numbers as $\mathcal{B}_{\sigma} = \{ x \in (0,1) \setminus \mathbb{Q} : B_{\sigma}(x) < \infty \}.$

The $\sigma$-Brjuno function defined in \eqref{ps} satisfies the functional equation
\begin{equation}\label{FE}
B_{\sigma}(x)=x^{-1/\sigma}+xB_{\sigma}(A(x)) \text{ for  all } x\in(0,1)
\end{equation}
and more generally,
\begin{equation} \label{gen}
B_{\sigma}(x)=B_{\sigma}^{(K)}(x)+\beta_{K}(x)B_{\sigma}(A^{K+1}(x)) \quad (K\in \N, x\in(0,1)\setminus \Q),
\end{equation}
where $B_{K}$ denotes the partial sum with respect to the regular continued fraction, defined as
\begin{equation*}
B_{\sigma}^{(K)}(x)=\sum_{j=0}^{K}\beta_{j-1}{(x)}(1/A^{j}(x))^{1/\sigma}.
\end{equation*}

The operator $(1-T)$, where $T$ is the transfer operator of the Gauss map, relates $B_\sigma$ to its singular part $\phi(x) = x^{-1/\sigma}$ in the following way
$$
(1 - T) B_{\sigma}(x) = \phi(x), \quad \text{where } \phi(x) = x^{-1/\sigma}.
$$
From this representation, several basic properties follow.
\begin{pro}[Integrability]
For $\sigma > 1$, we have $B_\sigma \in L^p(0,1)$ for all $1 \leq p < \sigma$. In particular, $B_2 \in L^1(0,1)$.
\end{pro}
\begin{proof}
Since $\phi(x) = x^{-1/\sigma} \in L^p(0,1)$ for all $p < \sigma$, the claim follows from the regularizing property of the operator $(1-T)^{-1}$ (see \cite[Proposition 10]{LuMar_10}).
\end{proof}

\begin{rem}[Lower semi-continuity]\label{rem:LSC}
The function $B_\sigma$ is lower semi-continuous and $1$-periodic, it attains a global minimum on $[0,1],$ (see section 3.4 of \cite{BCM})\end{rem}

\subsubsection{Relation with Classical Diophantine Classes.}

To situate the $\sigma$-Brjuno condition within classical metric number theory, we recall the definition of Diophatine condition. Let $\gamma>0$ and $\tau\geq 0$ be two real numbers. An irrational number $x\in \R\setminus \Q$ is Diophantine of exponent $\tau$ and constant $\gamma$ if,  for all $p,q \in \Z,$ $q>0$ it satisfies
$$\left| x - \frac{p}{q} \right| \geq \frac{\gamma}{q^{2+\tau}}.$$
Let  $\mathrm{CD}(\gamma,\tau)$ denote the set of all such irrational numbers. We define the standard Diophantine classes as: 
$$\mathrm{CD}(\tau)=\bigcup_{\gamma>0}  \mathrm{CD}(\gamma,\tau)
\quad \text{ and } \quad \mathrm{CD}=\bigcup_{\tau\geq 0}\mathrm{CD}(\tau).$$
In terms of continued fractions, $x \in \mathrm{CD}(\tau)$ corresponds to the growth condition $q_{n+1} = O(q_n^{1+\tau})$. Notably, $\mathrm{CD}(0)$ is the set of badly approximable numbers, characterized by bounded partial quotients,  while the complement in $\R\setminus\Q$ of $\mathrm{CD}$ is the set of  Liouville numbers.
The set
of Liouville numbers has zero Lebesgue measure and zero Hausdorff dimension.
The sets $\mathrm{CD}(\tau)$ and $  \mathrm{CD}$ are both PGL $(2,\Z)$ invariant. Moreover if $\tau > 0$
then $\mathrm{CD}(\tau)$ has full Lebesgue measure. The same holds for $\bigcap_{\tau>0} \mathrm {CD}(\tau)$ (Roth
numbers).

The following remark establishes that $\sigma$ serves as a sharp threshold for these classes:
 \begin{rem}[Diophantine characterization of $\mathcal{B}_\sigma$]
 For any $\sigma > 0$,
 $$
\bigcup_{\tau < \sigma} \mathrm{CD}(\tau) \;\subset\; \mathcal{B}_\sigma \;\subset\; \bigcap_{\tau\ge\sigma}\mathrm{CD}(\tau).
 $$
 In other words:
(1). if $x \in \mathrm{CD}(\tau)$  then $B_\sigma(x) < \infty,$
for all $\sigma > \tau.$
  
  (2). On the other hand,  if $B_\sigma(x) <+ \infty,$ then $x \in \mathrm{CD}(\sigma).$\end{rem}

\begin{proof}
(1). Assume $x \in \mathrm{CD}(\tau)$ with $\tau < \sigma$. Using $q_{n+1} = O(q_n^{1+\tau})$ and the estimate $\beta_{n-1} \asymp 1/q_n$, we obtain
$$
B_\sigma(x) \asymp \sum_{n=0}^\infty \frac{q_{n+1}^{1/\sigma}}{q_n^{1+1/\sigma}}
\le C \sum_{n=0}^\infty \frac{q_n^{(1+\tau)/\sigma}}{q_n^{1+1/\sigma}}
= C \sum_{n=0}^\infty \frac{1}{q_n^{1 - \tau/\sigma}}.
$$
Since $1 - (\tau/\sigma) > 0$ and $q_n$ grows at least exponentially, the series converges.

(2). Conversely, if $B_\sigma(x) < \infty$, then the general term of the series must tend to zero:
$$
\frac{q_{n+1}^{1/\sigma}}{q_n^{1+1/\sigma}} \to 0 \quad \text{as } n \to \infty.
$$
Hence $q_{n+1} = o\bigl( q_n^{1+\sigma} \bigr).$ Consequently, $q_{n+1} = O\bigl( q_n^{1+\sigma} \bigr)$ which is equivalent to $x \in \mathrm{CD}(\sigma).$
\end{proof}
 This characterization allows for the immediate classification of algebraic irrationals by linking their degree to the $\sigma$ parameter.
 \begin{cor} 
According to celebrated Roth's theorem if $x$ is an algebraic number (regardless of degree), then $x \in \mathrm{CD}(\tau)$ for any $\tau > 0.$ Consequently, $B_\sigma(x) < +\infty$ for all $\sigma>0.$
\end{cor}
 For small values of $\sigma,$ such as $0 < \sigma \le 1,$ the condition $B_\sigma(x) < \infty$ becomes significantly more restrictive. However, the set of $\sigma$-Brjuno numbers remains robust in a measure-theoretic sense: for every $0 < \sigma \le 1,$ the set $\mathcal{B}_\sigma$ is non-empty, has Hausdorff dimension $1$ and contains all quadratic irrationals (badly approximable numbers); including the golden ratio $g = [0;\overline{1}]$.  This is easily seen, as
 for any $x$ with bounded partial quotients, the denominators $q_n$ grow at least exponentially. Even when $\sigma$ is small, the exponential growth of the denominators in the series $\sum \beta_{n-1} x_n^{-1/\sigma}$ ensures geometric convergence. 

With the convergence properties and Diophantine links established, we now turn to the core of our work: the study of the global minimizers of $B_\sigma$. While algebraic numbers    are guranteed to stay with in $\mathcal{B}_\sigma$, we must now determine which specific arithmetic structures minimize the value of the function.

\section{A priori estimates.}\label{APri}
We begin by establishing a global lower bound for $B_\sigma$ which will be essential to show that the minimum of $B_\sigma$ is confined in the interval $[0,1/(n+1)]$. The rough idea is that  the functional equation \eqref{FE} implies that $B_\sigma$ satisfies the following simpler relation 
\begin{equation}\label{SFE}
B(x)\geq x^{-1/\sigma}+x \min_{ y\in [0,1]} B(y).
\end{equation}
It is not difficult to find a solution $g$ for equation \eqref{SFE} which is optimal (in the sense that $g$ satisfies \eqref{SFE} with an equal sign); $g$ is a lower bound for $B_\sigma$ and it is the best lower bound we can get just using \eqref{SFE}.

\begin{pro}
   Let $b^*:=\frac{(\sigma+1)^{1+1/\sigma}}{\sigma}$ and $g(x):= x^{-1/\sigma}+b^* x$. Then
    
    $$B_\sigma(x)\geq g(x)  \ \ \ \ \forall x \in [0,1], \ \ \ \ \ \min_{x\in [0,1]} B_\sigma(x) \geq \min_{x\in [0,1]} g(x)=b^*.$$
\end{pro}
\begin{proof}
    From the functional equation \eqref{FE}  we immediately get that, if $b\leq \min_{x\in [0,1]} B_\sigma(x)$ then
    $$B_\sigma(x)\geq x^{-1/\sigma}+b x.$$
    So, if $b<b^*$ is lower bound for $B_\sigma$ we can get a better lower bound $b'\in (b,b^*) $  as follows:  calling $f_b(x):=x^{-1/\sigma}+b x$,
    $$\min_{x\in [0,1]} B_\sigma(x)  \geq \min_{x\in [0,1]} f_b(x)= f_b((b\sigma)^{-\frac{\sigma}{\sigma+1}})=(b\sigma)^{\frac{1}{\sigma+1}}(1+\sigma^{-1}):=b'.$$
    Using repeatedly this fact, and starting from the lower bound $b_1=1$, we get an infinite sequence of lower bounds $b_k$:
    $b_{k+1}= \varphi(b_k)$ with $\varphi(x)=(x\sigma)^{\frac{1}{\sigma+1}}(1+\sigma^{-1}).$ It is easy to check that $\varphi(b)>b$ if $b<b^*$, hence  $b_k$ is strictly increasing; moreover $\lim_{k\to\infty} b_k=b^*$ where $\varphi(b^*)=(b^*)$ is the unique attractive fixed point of $\varphi$. 
\end{proof}

We can also use \eqref{gen} to improve the lower bound on the cylinder $(1/(k+1),1/k)$:
\begin{equation}\label{cylinder-bound}
    B_\sigma(x) \geq g_{k}(x), \ \ \ \mbox{ with } \ \ g_{k}(x)=x^{-1/\sigma} + x(\frac{1}{x}-k)^{-1/\sigma} + (1-kx)b^*.
\end{equation}

Note that $g_k-g$ is convex on $(\frac{1}{k+1}, \frac{1}{k})$; moreover setting $p=\frac{1}{k+\frac{1}{\sigma+1}}$ one checks that $g_k(p)-g(p)=g'_k(p)-g'(p)=0$, therefore 
\begin{equation}\label{gk}
g_k(x)\geq g(x) \quad \quad \forall x \in \left(\frac{1}{k+1}, \frac{1}{k}\right),
\end{equation}
which means that the new local lower bound $g_k$ on  always improves the global bound $g$ on its definition interval $(\frac{1}{k+1}, \frac{1}{k}).$

\begin{figure}
    \centering
    \includegraphics[width=0.5\linewidth]{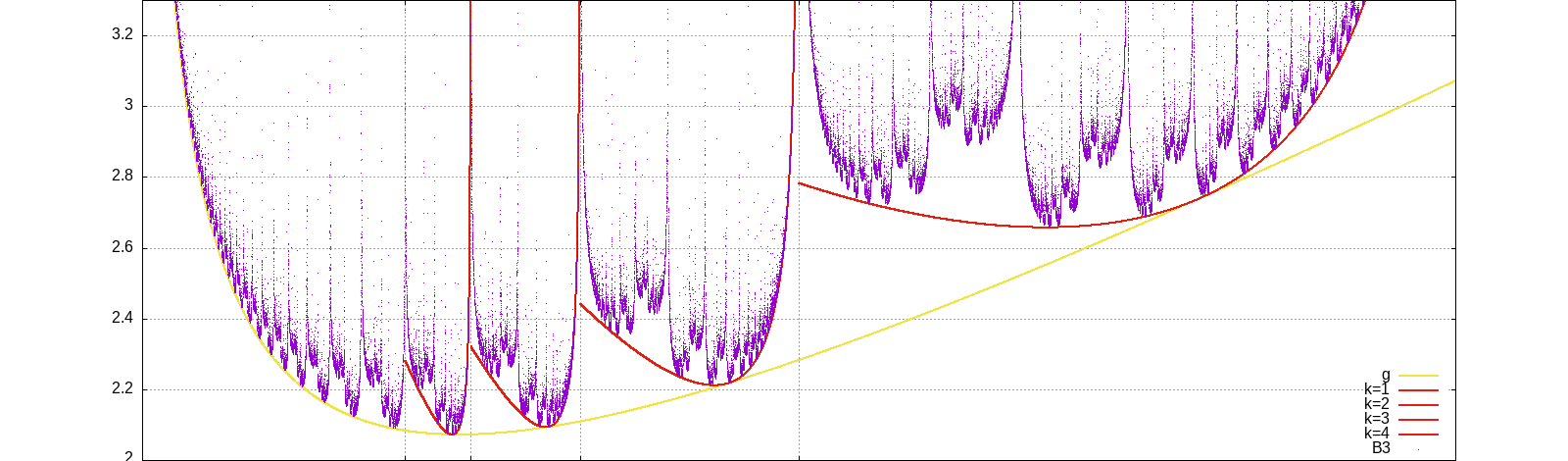}\label{bb}
    \caption{The function $B_3$ and the lower bounds $g$ (in yellow) and $g_k$ (in red) for $k=1,2,3,4$.}
    \label{fig:gvsgk}
\end{figure}
\subsection{Proof of Proposition \ref{Reg}.}

It turns out that for $\sigma=1$ the local lower bound \ref{cylinder-bound} is sufficient to establish Proposition \ref{Reg}: indeed a rigorous numerical computation shows that $$\min_{[1/2,1]} B_1(x)\geq \min_{[1/2,1]} g_1(x)>B_1(\eta_{n+1}).$$
In this case the computation is simple because point of minimum of $g_1$ is the root of the polynomial $-5x^4 + 10x^3 - 5x^2 + 2x - 1 = 0$; in principle an analogous strategy works also for other values of $n$, but the computations get heavier and heavier, and in any case a direct numerical computation cannot be the right way to establish the result for all $n\in \N$. 

For $n\geq 2$ let us set  $\xi_{n}:=\frac{1}{n+\frac{1}{n}}\in (\frac{1}{n+1}, 1)$. We claim that 
\begin{equation}\label{war}
\min_{x\in[\frac{1}{n},1]}B_{n}(x)\geq g(\xi_{n})>B_{n}(\eta_{n+1}).
\end{equation}
From this it follows that
\begin{equation}
    \min_{x\in[0,1]}B_{n}(x)=B_{n}(r) \Longrightarrow r\leq \frac{1}{n+1}.
\end{equation}
Thus, to complete the  proof of  Proposition \ref{Reg}, it suffices to establish Lemmas \ref{1L} and \ref{hard} below,  which correspond to the two required inequalities in \eqref{war}.
\begin{lem}\label{1L}
\begin{equation*}
B_{n}(x)\geq g(\xi_{n}) \quad \quad \forall x\in\left[\frac{1}{n+1},1\right].
\end{equation*}
\end{lem}
\begin{lem}\label{hard}
For all $n\in \N$, $n\geq 2$ we have
\begin{equation}\label{xieta}
    g(\xi_{n})>B_{n}(\eta_{n+1}).
\end{equation}
\end{lem}

\begin{proof}[Proof of Lemma \ref{1L}]
Since $B_{n}(x)\geq g(x) \ \ \forall x$  and $B_n(x) \geq g_n(x)  \  \forall  x\in [1/(n+1),1/n)$, we can  reduce the proof of Lemma \ref{1L} to the following two statements
\begin{eqnarray}\label{RL}
g(x)\geq g(\xi_{n}) & \forall x\in[\xi_n,1] \label{Rxi}\\
g_n(x)\geq g(\xi_{n}) & \forall x\in[\frac{1}{n+1},\xi_n]. \label{Lxi}
\end{eqnarray}
To prove \eqref{Rxi}, let us first note that $g$ is convex on $(0,1)$:
indeed $g''(x)=\frac{1}{n}\left(\frac{1}{n}+1 \right)x^{-\frac{1}{n}-2}>0$.
Therefore, since 
$g'(\xi_{n})=\frac{1}{n}\left( (n+1)^{1+\frac{1}{n}}-\left( (n+\frac{1}{n})\right)^{1+\frac{1}{n}}      \right)>0,$
    we immediately get our claim
    $$ g(x)\geq g(\xi_{n})+g'(\xi_{n})(x-\xi_n)\geq g(\xi_{n}) \ \ \  \forall x\geq \xi_{n}.$$

The proof of the second claim \eqref{Lxi} is slightly more tricky. In this case, since 
$(\frac{1}{n+1},\xi_{n}) \subset \left( \frac{1}{n+1},\frac{1}{n} \right)$, from \eqref{gk} we have $ g_n(x)\geq g(x)$  for all $x\in (\frac{1}{n+1},\xi_{n})$. On the other hand, since $g_{n}$ is convex we also know that $g_{n}(x)\geq g_{n}(\xi_{n})+g'_{n}(\xi_{n})(x-\xi_{n}).$
But 
\begin{align*}
g'_{n}(x)=-\frac{1}{n}x^{-\frac{1}{n}-1} +\left(\frac{1}{x}-n\right)^{-\frac{1}{n}}+\frac{1}{nx}\left(\frac{1}{x}-n \right)^{{-1/n-1}}-nb^{*}, \quad \text{ and}
\end{align*}
\begin{align*}
    {g_{n}'(\xi_{n})}&=-\left(1+\frac{1}{n^2} \right)\left(n+\frac{1} {n}\right)^{\frac{1}{n}}+n^{\frac{1}{n}}\left[1+n+\frac{1}{n}-(n+1)\left(1+\frac{1}{n}\right)^{\frac{1}{n}}\right]\\
    &=
-\left(1+\frac{1}{n^2} \right)\left(n+\frac{1} {n}\right)^{\frac{1}{n}}+n^{1+\frac{1}{n}}\left[1+\frac{1}{n}+\frac{1}{n^2}-\left(1+\frac{1}{n}\right)^{1+\frac{1}{n}}\right]  \\
&=
-\left(1+\frac{1}{n^2} \right)\left(n+\frac{1} {n}\right)^{\frac{1}{n}}+n^{1+\frac{1}{n}}\Psi\left( \frac{1}{n}\right),
\end{align*}
with 
$\Psi(x)=1+x+x^2-(1+x)^{1+x}.$ 

It is immediate to check that $\Psi(0)=\Psi'(0)=0$, moreover since $\Psi''(x) = 2 - (1+x)^{1+x}\left[(\ln(1+x)+1)^2 + \frac{1}{x+1}\right]$, and $\left[(\ln(1+x)+1)^2 + \frac{1}{x+1}\right]$  for $x>0$, we have $\Psi''(x)<0 \ \ 
 \forall x>0$, and therefore $\Psi(x)<0$ as soon as $x>0.$
    So we can conclude that $g_{n}'(\xi_{n})<-1$
    and by convexity inequality
    $ g_{n}(x)\geq g_{n}(\xi_{n})\geq g(\xi_n)$ for all $x\in\left[\frac{1}{n+1},\xi_{n} \right].$ 
\end{proof}

\begin{proof}[Proof of Lemma \ref{hard}] 
We want to prove the inequality
$g(\xi_{n})>B_{n}(\eta_{n+1})$ for $n\geq 2$; since both functions can be computed with arbitrary precision, on any finite set of integers we can easily check numerically that this inequality holds (in figure \ref{fig:numeric} we show the relation between $\delta_n:= g(\xi_{n})-B_{n}(\eta(n+1)) $ and $n$ for the first 800 integer values). 
\begin{figure}
    \centering
    \includegraphics[width=0.8\linewidth]{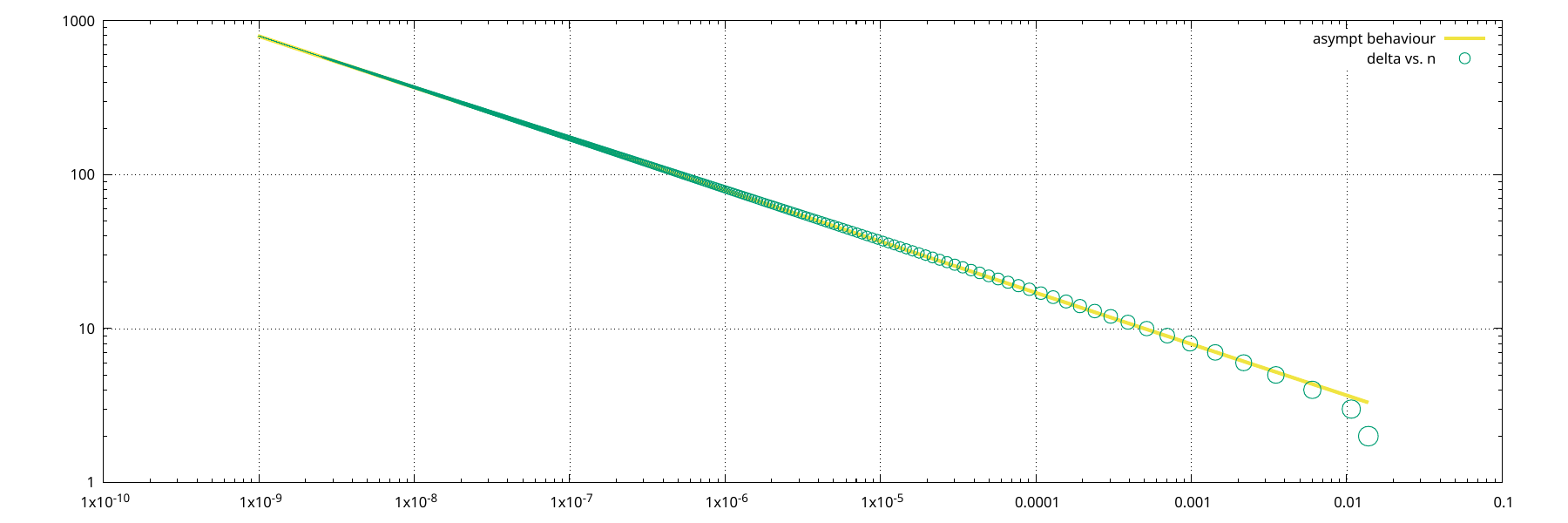}
    \caption{ The difference $\delta_n:= g(\xi_n)-B_n(\eta_{n+1})$ (on the $x$ axis) versus $n$ (on the $y$-axis) for $n\in [2,800]$. The graph is plotted in loglog scale, and the fact that for $n>>1$ these points appear to be aligned on a straight line of slope $-1/3$ derives from the asymptotic formula $\delta_n \sim \frac{1}{2}n^{-3}$ as $n\to +\infty$.}
    \label{fig:numeric}
\end{figure}

On the other hand, note that
\begin{align*}
   g(\xi_{n})&=\left( n+\frac{1}{n}\right)^{\frac{1}{n}}+\frac{b^*}{n+\frac{1}{n}}=n^{\frac{1}{n}} \left[   \left( 1+\frac{1}{n^2}\right)^{\frac{1}{n}}   +\frac{n+1}{n^2+1}\left(1+\frac{1}{n} \right)^{\frac{1}{n}}\right]\\
   &=n^{\frac{1}{n}}\gamma \left(\frac{1}{n}\right) \text{ where } \gamma \left(x\right):=(1+x^2)^x +\frac{x+x^2}{1+x^2}(1+x)^x.
   \end{align*}
On the other hand $\eta=\eta_{n+1}$ satisfies 
\begin{equation*}
\eta=\frac{1}{n+1+\eta} \quad \text{ and } \eta=\frac{\sqrt{(n+1)^2+4}-(n+1)}{2}.
\end{equation*}
By using functional equation \eqref{FE}, we have
$
B_{n}(\eta):=\frac{\eta^{-\frac{1}{n}}}{1-\eta}.$
Therefore,
\begin{align*}
B_{n}(\eta)&=(n+1+\eta)^{\frac{1}{n}}\frac{1}{1-\eta}==(n+1)^{\frac{1}{n}}\left( 1+\frac{\eta}{n+1}\right)^{\frac{1}{n}}\frac{1}{1-\eta}\\
&=n^{\frac{1}{n}}\left( 1+\frac{1}{n}\right)^{\frac{1}{n}} (1+\eta')^{\frac{1}{n}}\frac{1}{1-(n+1)\eta'},
\end{align*}
with
$\eta'=\frac{\eta}{n+1}=\frac{\sqrt{1+\frac{4}{(n+1)^2}}-1}{2}=\frac{\sqrt{1+\frac{\frac{4} {n^2}}{\left(1+\frac{1}{n}\right)^2}}-1}{2}.$
Therefore, we can write
$\eta'=r\left(\frac{1}{n} \right)$
with $r(x)=\frac{\sqrt{1+\frac{4x^2}{(1+x)^2}}-1}{2}$
and 
\begin{equation*}
    B_{n}(\eta)=n^{\frac{1}{n}}\beta\left(\frac{1}{n}\right) \text{ with } \beta(x)=(1+x)^x\cdot(1+r(x))^x \cdot \frac{x}{x-\left(1+{x}   \right)r(x)}.
\end{equation*}
Inequality \eqref{xieta} follows from the fact that it holds for $1\leq n \leq 10$ (by numerical evidence), and it also holds for $n>10 $ since
\begin{equation*}
    g(\xi_{n})-B_{n}(\eta)\geq n^{\frac{1}{n}}\left( \gamma \left( \frac{1}{n}\right)  -\beta \left( \frac{1}{n}\right)  \right)
\end{equation*}
and $w:=\gamma - \beta$ is positive on $(0,1/10)$;
indeed, this fact can be proved rigorously by a certified simple numerical scheme which we describe briefly. 

Let $F(x):=x-2w(x)/x^2$, since $w(x)\sim x^3/2$ as $x\to 0$ it is clear that $F$ extends to a $C^1$ function defined in a neighbourhood of the origin, such that $F'(0)=0$ and $F'(x)=1-2(x^{-2}w'(x)-2x^{-3}w(x))$. It is easy to check numerically that $F'(x)\leq 1/2$ for all $x\in J:=[0,1/10]$ i.e. $F:J \to J$ is a contraction, and the equation $F(x)=x$ has a unique solution on $J$, namely $x=0$. Since fixed points of $F$ are zeroes of $w$, this means that the only $x\in J$ where $w$  vanishes is $x=0$, hence it has constant (positive) sign on $(0, 1/10]$, thus $w(1/n)>0 $ also for all $n>10$.
\end{proof}

Note that $g$ depends smoothly on the parameter $\sigma$; to emphasize this dependence, we  denote it by
$g_{\sigma}$ in the following remark.
\begin{rem}\label{remU}
    There is a suitably small neighborhood $U_n$ of $n$ such that  \eqref{war} holds for $B_{\sigma}$ for all $\sigma \in U_{n}$ as well, namely 
    \begin{eqnarray*}
        \min_{x\in[\frac{1}{n},1]}B_{\sigma}(x)\geq g_{\sigma}(\xi_{n}),\\ g_{\sigma}(\xi_{n})>B_{\sigma}(\eta_{n+1}). 
    \end{eqnarray*}
\end{rem}

\section{Global minimum of $B_{\sigma}.$}\label{GMin}
As established in the previous section, a priori estimates allow us to narrow the domain of interest for the global minimum of $B_{\sigma}(x).$ Specifically, it suffices to analyze the function on the restricted interval $ [0, \frac{1}{n+1}]$, leading to  Proposition \ref{MT}.  

\subsection{Proof of Proposition \ref{MT}.}
Recall from the hypothesis of statement of Proposition \ref{MT} that we have 
 \begin{equation}\label{C2}
\sigma<n+\eta_{n+1}
\end{equation} 
 and
\begin{equation}\label{C1}
\min_{x\in[0,1]}B_{\sigma}(x)=B_{\sigma}(r) \text{ with }  r\leq \frac{1}{n+1}.
\end{equation}
Then  our aim is to show
$
r=\eta_{n+1}.$

Define $C:=\inf_{x\in[0,1]} B_{\sigma}(x).$ 
By Remark \ref{rem:LSC}, $B_{\sigma}$ is lower semi-continuous; thus the infimum is attained at some $r\in(0,1)$ such that 
$C=B_{\sigma}(r).$ Based on the a priori argument established previously, we know $r\in(0,\frac{1}{n+1}).$ This yields the following corollary.
\begin{cor}
$C=B_{\sigma}(r)$ for some $r\in(0,\frac{1}{n+1}).$ 
\end{cor}
To complete the proof of Proposition \ref{MT}, we must show that minimizer $r$ is exactly the fixed point $\eta_{n+1}.$
We will establish that $r \geq \eta_{n+1}$ and $r \leq \eta_{n+1},$ separately. This is achieved through a series of propositions that provide iterative lower bounds and exact identities coming from the for the functional equation.

We begin by establishing a general lower bound for the minimum value $C$.
\begin{pro}\label{P3s}
Let $r\in(0,\frac{1}{n+1}).$ Then for all $K\in \N,$ we have 
\begin{equation*}
C=B_{\sigma}(r)\geq \frac{B^{(K)}_{\sigma}(r)}{1-\beta_{K}(r)}.
\end{equation*}
\end{pro}
\begin{proof}
By using \eqref{gen},  we have
\begin{align*}
C=B_{\sigma}(r)=B^{(K)}_{\sigma}(r)+\beta_{K}(r)B_{\sigma}(A^{K+1}(r)) &\geq  B_{\sigma}^{(K)}(r)+C\beta_{K}(r).
\end{align*}
Therefore $C\geq\frac{B^{(K)}_{\sigma}(r)}{1-\beta_{K}(r)}.$
\end{proof}
Next, we observe that for the fixed point $\eta_{n+1},$ this inequality becomes an exact identity.
\begin{pro}\label{P4s}
For all $K\in \N,$ we have 
\begin{equation*}
B_{\sigma}(\eta_{n+1})= \frac{B^{K}_{\sigma}(\eta_{n+1})}{1-\beta^{K}_{\sigma}(\eta_{n+1})}.
\end{equation*}
\end{pro}
\begin{proof}
Again by using \eqref{gen}, now for $r=\eta_{n+1}$ we obtain, 
\begin{align*}
B_{\sigma}(\eta_{n+1})=B^{K}_{\sigma}(\eta_{n+1})+\beta_{K}(\eta_{n+1})B_{2}(A^{K+1}(\eta_{n+1})),
\end{align*}
since $\eta$ is the fixed point of the Gauss map $A$ i.e. $A^{K+1}(\eta_{n+1})=A(\eta_{n+1})=\eta_{n+1}$ for any $K\in\N,$ therefore the required result follows.
\end{proof}

\begin{pro}
Let $r\in(0,\frac{1}{n+1})$ such that $C=B_{\sigma}(r).$ Then $\eta_{n+1}\leq r.$
\end{pro}
\begin{proof}
Let $h_{\sigma}(x):= \frac{x^{-1/\sigma}}{1-x}$; from Proposition \ref{P3s} and \ref{P4s} with $K=0$ and the definition of $C,$ we have
\begin{equation}\label{utile}
h_{\sigma}(\eta_{n+1})=\frac{B^{(0)}_{\sigma}(\eta_{n+1})}{1-\beta_{0}(\eta_{n+1})}=B_{\sigma}(\eta_{n+1})\geq C= B_{\sigma}(r)\geq \frac{B^{(0)}_{\sigma}(r)}{1-\beta_{0}(r)}= h_{\sigma}(r).
\end{equation}

Next $0<x< 1/(n+1),$ 
\begin{align*}
\frac{B^{(0)}_{\sigma}(x)}{1-\beta_{0}(x)}=\frac{x^{-1/\sigma}}{1-x}.
\end{align*}
 We have 
$h'_{\sigma}(x)=\frac{(1+\sigma)x-1}{\sigma(1-x)^2x^{1+1/\sigma}}.$ 
  Hence $h_{\sigma}$ is strictly decreasing for $x<1/(\sigma+1)$; Since by hypothesis $\sigma \leq n+\eta$ we deduce that  $1/(\sigma+1)\geq 1/(n+\eta_{n+1}+1)= \eta_{n+1}$, therefore $h_{\sigma}$ is decreasing on $[0,\eta_{n+1}] \subset [0,1/(\sigma+1)] .$  Suppose, for contradiction, that $r<\eta_{n+1}$. Then $r<\eta_{n+1}<1/(\sigma+1)$, so both $r$ and $\eta_{n+1}$ belong to the interval where $h_{\sigma}$ is decreasing. Consequently $h_{\sigma}(r)>h_{\sigma}(\eta_{n+1})$, contradicting $h_{\sigma}(\eta_{n+1})\ge h_{\sigma}(r)$ obtained in equation \ref{utile}. Therefore $r\ge\eta_{n+1}$.

\end{proof}

\begin{pro}
Let $r\in(0,\frac{1}{n+1})$ such that $C=B_{n}(r).$ Then $\eta_{n+1}\geq r.$
\end{pro}
\begin{proof}
We need to show $\eta_{n+1}\geq r.$ Let $f_\sigma(x):=\frac{B^{(1)}_{n}(x)}{1-\beta_{1}(x)}$; from Proposition \ref{P3s} and Proposition \ref{P4s} with $K=1$ and by the definition of $C,$ we have
\begin{align*}
f_\sigma(\eta_{n+1})=\frac{B^{(1)}_{\sigma}(\eta_{n+1})}{1-\beta_{1}(\eta_{n+1})}=B_{\sigma}(\eta_{n+1})\geq C=B_{\sigma}(r)\geq \frac{B^{(1)}_{\sigma}(r)}{1-\beta_{1}(r)}=f_\sigma(r).
\end{align*}
Note that for  $\frac{1}{n+2} <x<\frac{1}{n+1},$ $A(x)=\frac{1}{x}-(n+1).$
Therefore, 
\begin{align*}
f_{\sigma}(x)&=\frac{x^{-1/\sigma}+x(A(x))^{-1/\sigma}}{1-xA(x)}=     \frac{x^{-1/\sigma}+x(\frac{1}{x}-(n+1))^{-1/\sigma}}{(n+1)x}\\
&= \frac{1}{n+1}[x^{-1-1/\sigma}+(\frac{1}{x}-(n+1))^{-1/\sigma}].
\end{align*}
Both terms inside the parentheses are convex on $\left(\frac{1}{n+2},\frac{1}{n+1}\right)$, moreover
\begin{align*}
(n+1)f'_{\sigma}(x)=\frac{1}{\sigma}[x^{-2}(\frac{1}{x}-(n+1))^{-1-1/\sigma}-{(\sigma+1)}x^{-2-1/\sigma}].
\end{align*}

Next we show that 
$f_{\sigma}'(\eta_{n+1})> 0$ for $n\in\N.$
Recall that 
$\eta_{n+1}=\frac{1}{1+n+\eta_{n+1}}.$
\begin{align*}
(n+1)f_{\sigma}'(\eta_{n+1})&=
\frac{1}{\sigma}[\eta_{n+1}^{-2}(\frac{1}{\eta_{n+1}}-(n+1))^{-1-1/\sigma}-{(\sigma+1)}\eta_{n+1}^{-2-1/\sigma}]\\
&=\frac{1}{\sigma}\eta_{n+1}^{-2-1/\sigma}[\frac{1}{\eta}-(\sigma+1)]=\frac{1}{\sigma}\eta_{n+1}^{-2-1/\sigma}[(n
+1+\eta_{n+1})-{(\sigma+1})].
\end{align*}
The term 
$(n
+1+\eta_{n+1})-{(\sigma+1)}>0,$ since we assumed that  \eqref{C2} holds, so we can write $\frac{1}{\eta_{n+1}}={n+1+\eta_{n+1}}>{\sigma+1}.$  Therefore $f_{\sigma}'(\eta_{n+1}) > 0$. The condition \eqref{C1} (from Proposition \ref{MT}) ensures that the global minimizer $r$ is a priori localized within the interval $(0, \frac{1}{n+1}]$. Since $f_{\sigma}$ is strictly convex on $(\frac{1}{n+2}, \frac{1}{n+1}]$ and $f_{\sigma}'(\eta_{n+1}) > 0$, $f_{\sigma}$ is strictly increasing for all $x \in [\eta_{n+1}, \frac{1}{n+1}]$. Suppose, for contradiction, that $r > \eta_{n+1}$. Then $f_{\sigma}(r) > f_{\sigma}(\eta_{n+1})$, which contradicts our established bound $f_{\sigma}(\eta_{n+1}) \geq f_{\sigma}(r)$. Thus, we must have $r \leq \eta_{n+1}$.
\end{proof}

\subsection{Proof of Theorem \ref{MTV2}}
We must show that there exits a neighbourhood $U_{n}$ of $n$ such that for all $\sigma\in U_{n}$
\begin{equation*}
\min_{[0,1]}B_{\sigma}(x)=B_{\sigma}(\eta_{n+1}).
\end{equation*} 

For $\sigma=n,$ Theorem \ref{MTV} shows that the minimizer is unique and equals $\eta_{n+1}$. 
By Remark \ref{remU} there exists a nbd $U_n$ such that $g_\sigma(\xi_n)\leq \min_{x\in [1/(n+1),1]} B_\sigma(x)$ and $g_{\sigma}(\xi_{n})>B_{\sigma}(\eta_{n+1})$ for all $\sigma \in U_n$.
Since we can also assume that $U_n\subset (0,n+\eta_{n+1})$ for all $\sigma\in U_n$, Proposition \ref{MT} implies that the unique minimizer of $B_{\sigma}$ is $\eta_{n+1}$.

\section{Scaling properties of $B_{n}$ for $n\ge 2$.}\label{SPn}

Having identified the unique global minimum at $\eta_{n+1}$, we now investigate the local behavior of $B_{n}$ with near this point. We establish that for any integer $n \ge 2$, the function $B_{n}$ exhibits square-root scaling, which is a signature of the ``cusp-like" singularity typically found in Brjuno-type functions.

\begin{thm}
For each $n \ge 2$, there exists a constant $c_n > 0$ such that
\begin{equation*}
B_n(x) - B_n(\eta_{n+1}) \ge c_n |x - \eta_{n+1}|^{1/2} \qquad \forall x \in \bigl[0, \tfrac{1}{n+1}\bigr],
\end{equation*}
where $\eta_{n+1} = [0;\overline{n+1}] = \frac{\sqrt{(n+1)^2+4}-(n+1)}{2}$.
\end{thm}

The proof for general $n$ follows from a localization argument in the first cylinder. For the sake of clarity, we provide the detailed proof for the case $n=2$. The arguments for $n > 2$ are analogous, with the primary adjustment being the refinement of the initial interval $[0, x_1]$, where for general $n$, we take $x_1 = \frac{n+1}{(n+1)^2 + 1}$.

\subsection{Scaling properties of $B_{2}.$}\label{SP2}

We sketch the argument for $n=2$, the proof of general $n$ can be expressed on similar lines.
\begin{thm}
\begin{equation*}
B_{2}(x)-B_{2}(\eta_{3})\geq c|x-\eta_{3}|^{1/2}.\end{equation*}
\end{thm}

\begin{proof}
Let $\Psi(x)=\frac{1}{3+x}$ and define the recursive relation 
\begin{equation}\label{RRR}
x_{n+1}=\Psi(x_{n})   \quad \text{ where } x_{1}\in (0,3/10].
\end{equation}
 
Then from \eqref{FE}, we have 
$
B_{2}(\Psi(x_{n}))=(\Psi(x_{n}))^{-1/2}+\Psi(x_{n})B_{2}(x_{n}) 
$ and 
$
B_{2}(\Psi(\eta_{3}))=(\Psi(\eta_{3}))^{-1/2}+\Psi(\eta_{3})B_{2}(\eta_{3})= \eta_{3}^{-1/2}+\eta_{3} B_{2}(\eta_{3}),
$
last equality follows from the fact that $\Psi(\eta_{3})=\eta_{3}.$ Now
\begin{align*}
B_{2}(x_{n+1})-B_{2}(\eta_{3})=x_{n+1}^{-1/2}-\eta_{3}^{-1/2}+x_{n+1}[B_{2}(x_{n})-B_{2}(\eta_{3})]+B_{2}(\eta_{3})[x_{n+1}-\eta_{3}].
\end{align*}

Let $h(x):=x^{-1/2}$, then By Taylor's expansion     
\begin{align*}x_{n+1}^{-1/2}-\eta_{3}^{-1/2}&\approx h'(\eta_{3}) (x_{n+1}-\eta_{3}) +\frac{h''(\eta_{3})}{2} (x_{n+1}-\eta_{3})\geq -\frac{1}{2}\eta_{3}^{-3/2}(x_{n+1}-\eta_{3}),
\end{align*}
as $\frac{h''(\eta_{3})}{2} (x_{n+1}-\eta_{3})>0.$

Setting $\mathcal E_{n}:=B_{2}(x_{n})-B_{2}(\eta_{3})$ and $\delta_{n}:=x_{n}-\eta_{3}$ with Sign$(\delta_{n})=(-1)^{n},$ we get 
\begin{equation*}
\mathcal E_{n+1}\geq x_{n+1}\mathcal E_{n}-l \delta_{n+1}
\end{equation*}
where $l:=\frac{1}{2}\eta_{3}^{-3/2}-B_{2}(\eta_{3})$ and as $  \frac{1}{2}\eta_{3}^{-3/2}    >  B_{2}(\eta_{3})$ therefore $l> 0.$

Iterating this relation we get 
$
\mathcal E_{n+1}\geq x_{n+1}x_{n}\mathcal E_{n-1}-l (x_{n+1}\delta_{n}+\delta_{n+1}).
$
Note that 
\begin{align*}
x_{n+1}\delta_{n}+\delta_{n+1} &=x_{n+1}(x_{n}-\eta_{3})+x_{n+1}-\eta_{3}   \\
&=x_{n+1}(1+x_{n})-\eta_{3}(1+x_{n+1}),\\
&= x_{n+1}(3+x_{n})-\eta_{3}(3+x_{n+1})-2(x_{n+1}-\eta_{3})=(3+x_{n})\delta_{n+1}-2\delta_{n} .
\end{align*}
The term  $ (3+x_{n})\delta_{n+1}-2\delta_{n}         $ has sign $(-1)^{n}.$
Therefore, for even values of $n$ we have 
$
\mathcal E_{n+2}\geq x_{n+2}x_{n+1}\mathcal E_{n}.
$

Let $\bigcup_{n=0}^{\infty}\Psi^{n}([0,\frac{3}{10}])=[0,\frac{3}{10}]\setminus \eta_{3}.$
For all $n\ge 1,$ 
$
\mathcal E_{2n+2}\geq \lambda_{n} \mathcal E_{2n},$
where 
$\lambda_{n}:=x_{2n+2}x_{2n+1}.$
Note that from \eqref{RRR}, we have  $x_{2}=\frac{1}{3+x_{2}},x_{3}=\frac{3+x_{1}}{10+3x_{0}},x_{4}=\frac{10+3x_{1}}{33+10x_{1}}, \cdots.$ Continuing in this way \eqref{RRR} induces the recursive relation
\begin{equation*}
x_{n}=\frac{G_{n-1}+x_{1}H_{n-1}}{G_{n}+x_{1}H_{n}}, \quad \forall {n\geq 2} \text{ and } x_{1}\in\left[0,\frac{3}{10}\right],
\end{equation*}
where $G_{n+1}=3G_{n}+
G_{n-1}$ are with $G_{1}=1$ and $G_{2}=3$ and 
$H_{n+1}=3H_{n}+
H_{n-1}$ are with $H_{1}=0$ and $H_{2}=1.$

 Therefore, we can write
$
\lambda_{n}=\frac{G_{2n}+x_{1}H_{2n}}{G_{2n+2}+x_{1}H_{2n+2}}.
$ Further
for all $n\ge1$
$
\frac{G_{2n}}{G_{2n+2}}\leq \lambda_{n} \leq \frac{G_{2n+2}}{G_{2n+4}},$
 and $\lambda_{n}\to \eta^2_{3}$ as $n\to\infty$ (exponentially).

In order to obtain the optimal estimate first note that
\begin{align*}
\mathcal E_{2n+2}&=\big( \prod^n_{k=1}\lambda_{k} \big) \mathcal E_{2} \text{ with} 
\prod^n_{k=1}\lambda_{k} =\exp (\sum^{n}_{k=1}\log \lambda_{k})=\exp (n\log \eta_{3}^{2}+\sum^{n}_{k=1}\log \frac{\lambda_{k}}{\eta_{3}^2})=\eta_{3}^{2n}\exp (\sum^{n}_{k=1}\log \frac{\lambda_{k}}{\eta_{3}^2}).
\end{align*}
The term $\sum^{\infty}_{k=1}\log \frac{\lambda_{k}}{\eta_{3}^2}$ in last equation is an absolutely convergent series because $\lambda_{k} \to \eta_{3}^{2}$ act an exponential rate,
therefore we can write
\begin{align*}
\mathcal E _{2n+2}\geq \sigma \mathcal E _{2} \eta_{3}^{2n} \quad \text{ where }\quad \sigma:= \exp (-\sum^{\infty}_{n=1}\log ( \frac{G_{2n+2}}{G_{2n}})).
\end{align*}
Thus, we can write
\begin{equation} \label{EstB}
\mathcal E_{2n+2}\geq c^* \eta_{3}^{2n} \text{ for some constant } c^* .
\end{equation}
This estimate shows that $B_{2}$ has a cusp like minima at $\eta_{3},$ namely 
\begin{equation*}
B_{2}(x)-B_{2}(\eta_{3})\geq c|x-\eta_{3}|^{\tau}        \quad \text{ with }
\end{equation*}
 $\tau=\frac{1}{2}$
 where the value of $\tau$ is obtained by using the fact that 
$\delta_{2n+2}\asymp \eta_{3}^{4n+4}$ for any $n\geq1$ and comparing it with the estimates \eqref{EstB}.
\end{proof}

\begin{figure}[h]
 \centering
  \includegraphics{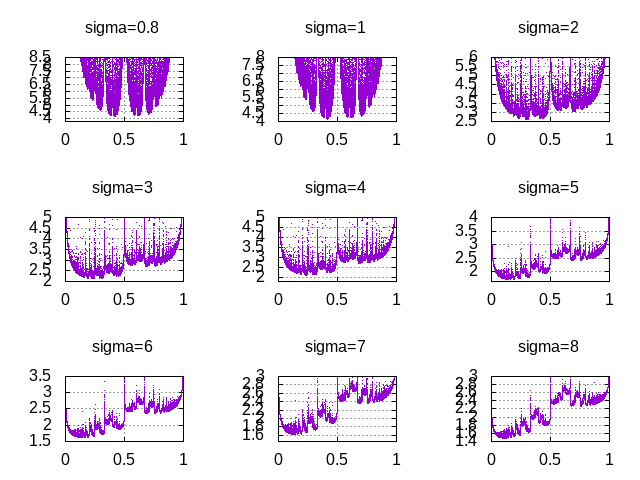}\label{mfig}
 \caption{ Graphs of $B_{\sigma}$ for different values of $\sigma$.}   \end{figure}

%
%

\begin{thebibliography}{10}

\bibitem{BCM}
A.~Bakhtawar, C.~Carminati, and S.~Marmi.
\newblock Global and local minima of {$\alpha$}-{B}rjuno functions.
\newblock {\em Monatsh. Math.}, 207(2):197--230, 2025.

\bibitem{BaMa_12}
M.~Balazard and B.~Martin.
\newblock Comportement local moyen de la fonction de {B}rjuno.
\newblock {\em Fund. Math.}, 218(3):193--224, 2012.

\bibitem{BaMa_20}
M.~Balazard and B.~Martin.
\newblock Sur le minimum de la fonction de {B}rjuno.
\newblock {\em Math. Z.}, 296(3-4):1819--1824, 2020.

\bibitem{Br_71}
A.~D. Brjuno.
\newblock Analytic form of differential equations. {I}, {II}.
\newblock {\em Trudy Moskov. Mat. Ob\v{s}\v{c}.}, 25:119--262; ibid. {\bf 26
  (1972), 199--239}. 
  {377192}, 

\bibitem{BC06}
X.~Buff and A.~Ch\'eritat.
\newblock The {B}rjuno function continuously estimates the size of quadratic
  {S}iegel disks.
\newblock {\em Ann. of Math. (2)}, 164(1):265--312, 2006.

\bibitem{ChMa_22}
C.~Chavaudret and S.~Marmi.
\newblock Analytic linearization of a generalization of the semi-standard map:
  radius of convergence and {B}rjuno sum.
\newblock {\em Discrete Contin. Dyn. Syst.}, 42(7):3077--3101, 2022.

\bibitem{Da_94}
A.~M. Davie.
\newblock The critical function for the semistandard map.
\newblock {\em Nonlinearity}, 7(1):219--229, 1994.

\bibitem{StMa_18}
S.~Jaffard and B.~Martin.
\newblock Multifractal analysis of the {B}rjuno function.
\newblock {\em Invent. Math.}, 212(1):109--132, 2018.

\bibitem{LuMar_10}
L.~Luzzi, S.~Marmi, H.~Nakada, and R.~Natsui.
\newblock Generalized {B}rjuno functions associated to {$\alpha$}-continued
  fractions.
\newblock {\em J. Approx. Theory}, 162(1), 2010.

\bibitem{MA_00}
S.~Marmi.
\newblock An introduction to small divisors.
\newblock {\em Quaderni del Dottorato di Ricerca in Matematica, Pisa (2000)
  (also available at mp-arc and front.math.ucdavis.edu: math.DS/0009232)}.

\bibitem{Mar_90}
S.~Marmi.
\newblock Critical functions for complex analytic maps.
\newblock {\em J. Phys. A}, 23(15):3447--3474, 1990.

\bibitem{MaCa08}
S.~Marmi and C.~Carminati.
\newblock Linearization of germs: regular dependence on the multiplier.
\newblock {\em Bull. Soc. Math. France}, 136(4):533--564, 2008.

\bibitem{MaMoYo_01}
S.~Marmi, P.~Moussa, and J.-C. Yoccoz.
\newblock Complex {B}rjuno functions.
\newblock {\em J. Amer. Math. Soc.}, 14(4):783--841, 2001.

\bibitem{MaMoYo_06}
S.~Marmi, P.~Moussa, and J.-C. Yoccoz.
\newblock Some properties of real and complex {B}rjuno functions.
\newblock In {\em Frontiers in number theory, physics, and geometry. {I}},
  pages 601--623. Springer, Berlin, 2006.

\bibitem{MaJa_92}
S.~Marmi and J.~Stark.
\newblock On the standard map critical function.
\newblock {\em Nonlinearity}, 5(3):743--761, 1992.

\bibitem{MaYa_02}
S.~Marmi and J.-C. Yoccoz.
\newblock Some open problems related to small divisors.
\newblock In {\em Dynamical systems and small divisors ({C}etraro, 1998)},
  volume 1784 of {\em Lecture Notes in Math.}, pages 175--191. Springer,
  Berlin, 2002.

\bibitem{Si_42}
C.~L. Siegel.
\newblock Iteration of analytic functions.
\newblock {\em Ann. of Math. (2)}, 43:607--612, 1942.

\bibitem{Yo_95}
J.-C. Yoccoz.
\newblock Th\'{e}or\`eme de {S}iegel, nombres de {B}runo et polyn\^{o}mes
  quadratiques.
\newblock Number 231, pages 3--88. 1995.
\newblock Petits diviseurs en dimension $1$.

\bibitem{Yo_02}
J.-C. Yoccoz.
\newblock Analytic linearization of circle diffeomorphisms.
\newblock In {\em Dynamical systems and small divisors ({C}etraro, 1998)},
  volume 1784 of {\em Lecture Notes in Math.}, pages 125--173. Springer,
  Berlin, 2002.

\end{thebibliography}

\Addresses

\end{document}